\documentclass[a4paper,10pt,reqno, english]{amsart}

\usepackage{amsmath,amssymb,amscd,amsthm,amsfonts}
\usepackage{graphicx,subfigure}
\usepackage[pdfencoding=auto]{hyperref}
\usepackage{dsfont}
\usepackage{xcolor}
\usepackage[utf8]{inputenc}
\usepackage{microtype}
\usepackage[alphabetic, nobysame]{amsrefs}

\newtheorem{theorem}{Theorem}[section]

\newtheorem{conjecture}[theorem]{Conjecture}

\newcommand{\rr}{\mathds{R}}

\DeclareMathOperator{\dist}{dist}
\title{Counterexamples to the Colorful Tverberg Conjecture for Hyperplanes}
\author{Jo\~ao Pedro Carvalho}
\address{Haverford College. Haverford, PA 19041}
\email{jdecarvalh@haverford.edu}
\author{Pablo Sober\'on}
\address{Baruch College, City University of New York.  New York, NY 10010}
\email{pablo.soberon-bravo@baruch.cuny.edu}
\thanks{This project was done as part of the 2021 New York Discrete Math REU, funded by NSF grant DMS 2051026.  Carvalho's research was supported by the KINSC Summer Scholar program by Haverford College.  Sober\'on’s research is supported by NSF grant DMS 2054419 and a PSC-CUNY TRADB52 award.}
\date{}

\begin{document}

\begin{abstract}
    In 2008 Karasev conjectured that for every set of $r$ blue lines, $r$ green lines, and $r$ red lines in the plane, there exists a partition of them into $r$ colorful triples whose induced triangles intersect.  We disprove this conjecture for every $r$ and extend the counterexamples to higher dimensions.
\end{abstract}

\maketitle

\section{Introduction}

Tverberg's theorem is a central result in combinatorial geometry.  It states that \textit{any set of $(r-1)(d+1)+1$ points in $\rr^d$ can be partitioned into $r$ sets whose convex hulls intersect.}  It was proven in 1966 by Helge Tverberg \cite{Tverberg:1966tb} and its generalizations and extensions have sparked significant interest in the area \cites{Blagojevic:2017bl, Barany:2018fy, DeLoera:2019jb}.

One of the most intriguing open questions regarding Tverberg's theorem is the colorful version.  This was conjectured by B\'ar\'any and Larman.

\begin{conjecture}[B\'ar\'any, Larman 1992 \cite{Barany:1992tx}]
	Let $r,d$ be positive integers.  Given $d+1$ sets $X_1, \ldots, X_{d+1}$, each consisting of $r$ points of $\rr^d$, there exists a partition of their union into $r$ sets $A_1, \ldots, A_r$ such that each $A_j$ has exactly one point of each $X_i$ and the convex hulls of $A_1, \ldots, A_r$ intersect.
\end{conjecture}

The conjecture above has been confirmed when $r+1$ is prime \cites{Blagojevic:2011vh, Blagojevic:2015wya}, or when $d=2$ \cite{Barany:1992tx}. If each $X_i$ is allowed to have more than $r$ points, there are plenty of positive results (see, e.g., \cites{Blagojevic.2014, Zivaljevic:1992vo} and the references therein).

One popular way to find variations of Tverberg's theorem is to replace the convex hull operator.  The topological versions of Tverberg's theorem are the most notorious example, in which one wants to show that any continuous mapping from an $(r-1)(d+1)$-dimensional simplex to $\rr^d$ has points covered by the image of $r$ pairwise disjoint faces \cites{Barany:1981vh, Oza87, Volovikov:1996up}.  In this case, one must either introduce conditions on the parameters, such as being prime powers, or use a larger number of points \cites{mabillard2014, Frick:2015wp, Avvakumov2019, Frick:2020uj}.  Yet, there are other ways to replace the convex hull.  A recent example is if we consider for any two points the ball centered at their midpoint that contains both points in its boundary \cites{Huemer:2019ji, Soberon2020}.  Huemer et al. showed that the natural extension of the colorful Tverberg theorem is valid of this setting in the plane.  Another way to replace the convex hull is with Tverberg's theorem for hyperplanes.

In the hyperplane version, we are given sets of hyperplanes in $\rr^d$ in general position instead of points.  We say that a family of hyperplanes in $\rr^d$ is in general position if the normal vectors of any $d$ hyperplanes are linearly independent and no $d+1$ hyperplanes coincide.  With these conditions, any $d+1$ hyperplanes uniquely define a simplex.

\begin{theorem}[Karasev 2008 \cite{Karasev:2008jl}]\label{thm:Karasev-monochromatic}
	Let $r$ be a prime power and $d$ be a positive integer.  Any family of $r(d+1)$ hyperplanes in $\rr^d$ in general position can be split into $r$ sets of $d+1$ hyperplanes each so that their induced simplices all intersect.
\end{theorem}

The case $d=2$ was proved earlier by Rousseeuw and Hubert without conditions on $r$ \cite{Rousseeuw1999}.  Karasev also proved further extensions of the theorem above \cite{Karasev:2011jv}.  In 2008 Karasev conjectured that a colorful version of the Tverberg-type theorems for hyperplanes should hold similar to the B\'ar\'any--Larman conjecture.

\begin{conjecture}[Karasev 2008 \cite{Karasev:2008jl}]\label{conj:Karasev2008}
	Given $r$ blue, $r$ red, and $r$ green lines in the plane in general position, it is possible to split them into $r$ colorful triples whose induced triangles intersect.
\end{conjecture}

If $r$ is a prime power, Karasev proved a relaxed version of the conjecture above if each family has $2r-1$ lines but we only look for $r$ colorful triples whose triangles intersect.  In this note, we disprove Conjecture \ref{conj:Karasev2008} for every $r$.  This was previously known to be false when $r=2m$ for odd $m$ \cite{lee2018}.  Our examples work in higher dimensions, so we also disprove the natural extension to $\rr^d$.  In other words, Theorem \ref{thm:Karasev-monochromatic} cannot be made colorful as in the B\'ar\'any--Larman conjecture.

\begin{theorem}\label{thm:main}
	Let $r, d \ge 2$ be positive integers.  There exists a family of $d+1$ sets $X_1, \ldots, X_{d+1}$, each consisting of $r$ hyperplanes in $\rr^d$, so that their union is in general position and the following statement holds.  For any partition of their union into $r$ sets $A_1,\ldots, A_r$ so that each $A_j$ has a hyperplane of each $X_i$, there are two set $A_j$, $A_{j'}$ whose induced simplices do not intersect.
\end{theorem}

\section{Counterexample}

We first construct the counterexample in dimension two, and extend our constructions in higher dimensional spaces.  Given $d+1$ sets of hyperplanes in $\rr^d$, we say that the are in \textit{weak general position} if every time we pick $d$ hyperplanes of different sets, their normal vectors are linearly independent.  This guarantees that their intersection is a single point.  

Our constructions are made for sets in weak general position, and a small perturbation argument gives Theorem \ref{thm:main}.  We consider each set $X_i$ as a color class.  Given $d+1$ sets $X_1, \ldots, X_{d+1}$ of $r$ hyperplanes each in $\rr^d$, we say that a colorful partition is a partition of their union into $r$ sets $A_1, \ldots, A_r$, each with one element of each $X_i$.  For a colorful $(d+1)$-tuple $A$ in weak general position we denote by $\triangle(A)$ the simplex generated by $A$.  If the hyperplanes of $A$ are concurrent, then $\triangle(A)$ is their point of intersection.

\begin{theorem}\label{thm:counterplane}
Let $r \ge 2$ be an integer.  Consider the set $X_1$ (red) of $r-1$ copies of the line $x=1$ and one copy of the line $x=-1$, the set $X_2$ (green) of $r-1$ copies of the line $y=-1$ and one copy of $y=1$, and the set $X_3$ (blue) of $r-1$ copies of the line $x=-y$ and one copy of the line $x=y$.  Then, every colorful partition $A_1, \ldots, A_r$ of $X_1, X_2, X_3$ has two sets $A_j, A_j'$ so that $\triangle(A_j) \cap \triangle (A_{j'}) = \emptyset$.
\end{theorem}

The construction can be seen in Figure \ref{fig:counterexample-plane}.

\begin{proof}
 To make a colorful partition of the three colors we first make a colorful partition of red and green and then assign blue lines to each part.  Choosing a pair of red and green lines is equivalent to choosing a vertex of the square $[-1,1]^2$.  The number of times a vertex can be chosen is bounded by the number of green lines containing it and by the number of red lines containing it.

If the point $(-1,-1)$ is chosen, then all the other $r-1$ vertices chosen have to be $(1,1).$ If the line associated with the point $(-1,-1)$ is $x=y,$ then all the other points are associated with the line $x=-y.$ So the first simplex is just the point $(-1,-1)$ that is not contained in the simplex defined by $(x=1,y=1,x=-y).$  If the blue line associated with $(-1,-1)$ is $x=-y$, then one of the copies of $(1,1)$ is associated with $x=y$.  That last simplex is simply the point $(1,1)$, which is not contained in the triangle spanned by $(x=-1,y=-1, x=-y)$.

If, on the other hand, the point $(-1,1)$ is chosen, then the point $(-1,-1)$ cannot be (they share the same red line). Similarly, the point $(1,1)$ can be chosen at most $r-2$ times (they share the same green line). Therefore the points in the partition have to be $(-1,1),$ $(1,-1),$ and $r-2$ copies of $(1,1).$ One of $(-1,1)$ or $(1,-1)$ will have to be associated with the line $x=-y,$ since there is only one copy of the line $x=y.$ This means that one of the simplices will be a point, say $(-1,1).$ But then one of the other simplices is either the point $(1,-1)$ or defined by this point and the line $x=y,$ and in either case these two simplices are disjoint. The case analysis is analogous if $(1,-1)$ is chosen.
\end{proof}

\begin{figure}
\centerline{\includegraphics{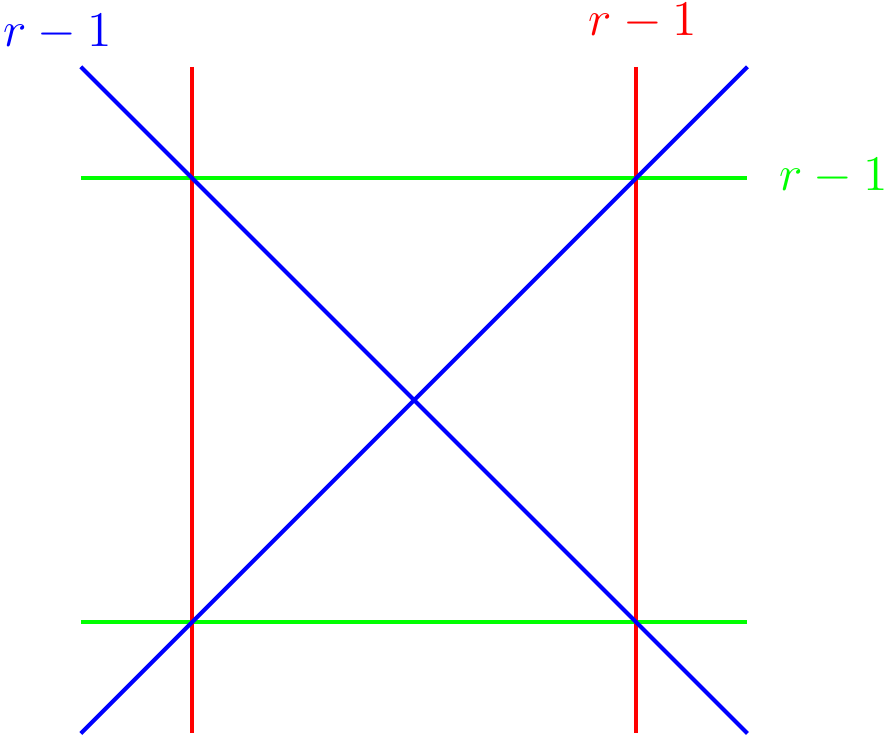} }
\caption{An illustration of the example for $d=2$.}
\label{fig:counterexample-plane}
\end{figure}

Now we extend this construction to $\rr^d$.  For point $x \in \rr^d$ we denote its coordinates by $x=(x_1, \ldots, x_d)$.

\begin{theorem}\label{thm:weak-general-pos}
	Let $d\ge 3 $, $r\ge 2$ be integers, and $\varepsilon >0$ be a real number.  Consider the sets $X_1, \ldots, X_{d+1}$  of hyperplanes in $\rr^d$ defined as follows.  For $i=1,\ldots, d$, the set $X_i$ is made by $r-1$ copies of the hyperplane $x_i = 1$ and one copy of the hyperplane $x_i = -1$.  The set $X_{d+1}$ is made by $r-1$ copies of the hyperplane $x_1+x_2 + \varepsilon(x_3 + \ldots + x_d)=0$ and one copy of the hyperplane $x_1 - x_2 + \varepsilon(x_3 + \ldots + x_d) = 0$.  Then, every colorful partition $A_1, \ldots, A_r$ of $X_1, \ldots, X_{d+1}$ has two sets $A_j, A_{j'}$ so that $\triangle(A_j) \cap \triangle (A_{j'}) = \emptyset$.
\end{theorem}

\begin{proof}
	Let $H$ be the two-dimensional plane spanned by $e_1, e_2$.  Let $\Pi:\rr^d \to H$ be the orthogonal projection.  Notice that every colorful partition of $X_1, X_2, \ldots, X_{d+1}$ can be made by first making a colorful partition of $X_1, X_2, X_{d+1}$ and then assign each of the hyperplanes in $X_i$ to different sets in the partition.  The colorful partition of $X_1, X_2, X_{d+1}$ intersected with $H$ is essentially the same construction as in Theorem \ref{thm:counterplane}.  We will show that there are two sets $A_j, A_j'$ such that $\Pi (\triangle(A_j)) \cap \Pi(\triangle (A_{j'}) = \emptyset$, which would prove the theorem.
	
	Let us compute the first two coordinates of every vertex of $A_j$.  Suppose $A_j$ has the hyperplane $x_1 = \lambda_1$ from $X_1$, the hyperplane $x_2 = \lambda_2$ from $X_2$, and the hyperplane $x_1 - \lambda_3 x_2 + \varepsilon (x_3 + \ldots + x_d) $ from $X_{d+1}$, and some hyperplane from each of $X_3, \ldots, X_d$.  Here $\lambda_1, \lambda_2, \lambda_3$ are all in $\{-1,1\}$.  To compute the first two coordinates of a vertex, we have to remove one of the $d+1$ hyperplanes of $A_j$ and find the intersection of the other $d$.  If neither the hyperplanes from $X_1$ nor $X_2$ are removed, the first two coordinates from the intersection are $(\lambda_1, \lambda_2)$.  If the hyperplane $x_1 = \lambda_1$ is removed, then the second coordinate is $\lambda_2$.  We also have $x_1 - \lambda_3 x_2 + \varepsilon(x_3 + \ldots + x_d) = 0$.  Therefore,
\[
|x_1 - \lambda_3 \lambda_2| < \varepsilon |x_3 + \ldots + x_d| \le \varepsilon (d-2).
\]
Therefore, the vertices of $\Pi(\Delta(A_j)$ are very close to the vertices of the triangle in the corresponding partition of $X_1, X_2, X_{d+1}$ in Theorem \ref{thm:counterplane}.  Since in Theorem \ref{thm:counterplane} we could always find two (possibly degenerate) disjoint triangles at distance greater than or equal to $\frac{1}{\sqrt{2}}$, then as long as $\varepsilon < \frac{1}{2\sqrt{2}(d-2)}$ we will have the same conclusion. 
\end{proof}

Finally, we make the approximation argument explicit to prove Theorem \ref{thm:main}.

\begin{proof}[Proof of Theorem \ref{thm:main}]
We take the construction from Theorem \ref{thm:weak-general-pos} and denote it $\Lambda$.  For each colorful partition $P$ of $\Lambda$, there exist two colorful triples $A_1, A_2$ so that $\triangle (A_1)$ and $\triangle (A_2)$ are strictly separated by a hyperplane $H$.  Let 
\[
\delta (P) = \min\{\dist(\triangle(A_1), H), \dist(\triangle(A_2), H)\}, \qquad \delta = \min_{P} \delta(P),
\]
where the minimum defining $\delta$ ranges over all colorful partitions of $\Lambda$.  We apply a small enough perturbation to $\Lambda$ so that for every colorful $d$-tuple $X$ of hyperplanes, the distance $\bigcap X$ moves is less than $\delta$.  This guarantees that the hyperplanes that separate colorful simplices still work.
\end{proof}

\section{Additional remarks}

Karasev's proof of Theorem \ref{thm:Karasev-monochromatic} relies on a clever geometric idea.  Given a set of $r(d+1)$ hyperplanes in $\rr^d$ in general position, and an additional point $p$, we can consider $S$ the set of projections of $p$ onto each hyperplane, and include $p$.  This gives us a set of $r(d+1)+1$ points in $\rr^d$, so Tverberg's theorem can be applied.  We obtain a partition of the point set into $r+1$ sets whose convex hulls intersect.  Karasev showed that if in this partition one of the parts is simply $\{p\}$ and the rest are $(d+1)$-tuples, then the partition they induce on the hyperplanes is the one we seek.  Moreover, for every $(d+1)$-tuple $A_j$ of hyperplanes we would have $p \in \triangle (A_j)$.  Then, the proof boils down to showing that for some point $p$ a partition as mentioned works.

At first glance, it might seem tempting to apply a similar idea by replacing Tverberg's theorem by the optimal colorful Tverberg theorem by Blagojevi\'c, Matschke, and Ziegler, as the number of points matches Tverberg's theorem.  One interpretation of the counterexamples found in this note is that controlling the size of the sets in an optimal colorful partition is significantly harder than for Tverberg's theorem, even with a similar degree of freedom in the configuration of points.  This is a bit surprising since the projection of $p$ onto a fixed hyperplane is an affine map, so the reason why Karasev's conjecture fails is not topological in essence.


\end{document}